\documentclass[a4paper,12pt,leqno]{article}
\usepackage{amsmath,amsfonts,amssymb,amsthm,graphicx,enumerate,mathrsfs,color,mathdots}

\newtheorem{theorem}{Theorem}[section]
\newtheorem{corollary}[theorem]{Corollary}
\newtheorem{algorithm}[theorem]{Algorithm}
\newtheorem{example}{Example}[section]

\newcommand{\ten}[1]{\mathcal{#1}}

\title{Fast Hankel Tensor-Vector Products and Application to Exponential Data Fitting}
\author{
Weiyang Ding\thanks{E-mail: dingw11@fudan.edu.cn. School of Mathematical Sciences, Fudan University, Shanghai, 200433, P. R. of China. This author is supported by 973 Program Project under grant 2010CB327900.}
\and
Liqun Qi\thanks{Email: liqun.qi@polyu.edu.hk. Department of Applied Mathematics, The Hong Kong Polytechnic University, Hung Hom, Kowloon, Hong Kong.
This author is supported by the Hong Kong Research Grant Council (Grant No. PolyU 502510, 502111, 501212, 501913).}
\and
Yimin Wei\thanks{E-mail: ymwei@fudan.edu.cn and yimin.wei@gmail.com. School of Mathematical Sciences and Key Laboratory of Mathematics for Nonlinear Sciences, Fudan University, Shanghai, 200433, P. R. of China. This author is supported by the National Natural Science Foundation of China under grant 11271084.}
}
\date{\today}

\begin{document}

\maketitle

\begin{abstract}
  This paper is contributed to a fast algorithm for Hankel tensor-vector products. For this purpose, we first discuss a special class of Hankel tensors that can be diagonalized by the Fourier matrix, which is called \emph{anti-circulant} tensors. Then we obtain a fast algorithm for Hankel tensor-vector products by embedding a Hankel tensor into a larger anti-circulant tensor. The computational complexity is about $\ten{O}(m^2 n \log mn)$ for a square Hankel tensor of order $m$ and dimension $n$, and the numerical examples also show the efficiency of this scheme. Moreover, the block version for multi-level block Hankel tensors is discussed as well. Finally, we apply the fast algorithm to exponential data fitting and the block version to 2D exponential data fitting for higher performance.

  \hspace{-16pt}{\bf Keywords:}
  Hankel tensor, block Hankel tensor, anti-circulant tensor, fast tensor-vector product, fast Fourier transformation, higher-order singular value decomposition, exponential data fitting.

  \hspace{-16pt}{\bf 2000 AMS Subject Classification:}
  15A69, 15B05, 65T50, 68M07.
\end{abstract}

\newpage

\section{Introduction}

Hankel structures arise frequently in signal processing \cite{Vadim2001}. Besides Hankel matrices, tensors with different Hankel structures have also been applied. As far as we know, the concept of ``Hankel tensor'' was first introduced in \cite{luque2003hankel} by Luque and Thibon.  Papy and Boyer et al. employed Hankel tensors and tensors with other Hankel structures in exponential data fitting in \cite{boyer2005delayed,papy2005exponential,papy2009exponential}. Then Badeau and Boyer discussed the higher-order singular value decompositions (HOSVD) of some structured tensors, including Hankel tensors, in \cite{badeau2008fast}. Recently, Qi investigated the spectral properties of Hankel tensor in \cite{qi2013hankel} largely via the generating function, Song and Qi proposed some spectral properties of Hilbert tensors in \cite{song2014some}, which are special Hankel tensors.

A tensor $\ten{H} \in \mathbb{C}^{n_1 \times n_2 \times \dots \times n_m}$ is called a \emph{Hankel tensor}, if
$$
\ten{H}_{i_1 i_2 \dots i_m} = \phi(i_1+i_2+\dots+i_m)
$$
for all $i_k = 0,1,\dots,n_k-1$ ($k = 1,2,\dots,m$). We call $\ten{H}$ a \emph{square Hankel tensor}, when $n_1=n_2=\dots=n_m$. Note that the degree of freedom of a Hankel tensor is $d_\ten{H}:=n_1+n_2+\dots+n_m-m+1$. Thus a vector ${\bf h}$ of length $d_\ten{H}$, which is called the \emph{generating vector} of $\ten{H}$, defined by
$$
h_k = \phi(k),\ k = 0,1,\dots,d_\ten{H}-1
$$
can totally determine the Hankel tensor $\ten{H}$, when the tensor size is fixed. Further, if the entries of ${\bf h}$ can be written as
$$
h_k = \int_{-\infty}^{+\infty} t^k f(t) {\rm d}t,
$$
we then  call $f(t)$ the \emph{generating function} of $\ten{H}$. The generating function of a square Hankel tensor is essential for studying eigenvalues, positive semi-definiteness,  and copositiveness etc. (cf. \cite{qi2013hankel}).

The fast algorithm for Hankel or Toeplitz matrix-vector products using fast Fourier transformations (FFT) is well-known (cf. \cite{chan2007introduction,ng2004iterative}). However the topics on Hankel tensor computations are seldom discussed. We propose an analogous fast algorithm for Hankel tensor-vector products and application to signal processing. In Section 2, we begin with a special class of Hankel tensors called \emph{anti-circulant} tensors. By using some properties of anti-circulant tensors, the Hankel tensor-vector products can be performed in complexity $\ten{O}(m d_\ten{H} \log d_\ten{H})$, which is described in Section 3. Then in Section 4, we apply the fast algorithm and its block version to 1D exponential data fitting and 2D exponential data fitting. Finally, we draw some conclusions and display some future research in Section 5.

\section{Anti-Circulant Tensor}

Circulant matrix is famous, which is a special class of  Toeplitz matrices \cite{chan2007introduction,ng2004iterative}. The first column entries of a circulant matrix shift down when moving right, as shown in the following $3$-by-$3$ example
$$
\begin{bmatrix}
c_0 & c_2 & c_1 \\
c_1 & c_0 & c_2 \\
c_2 & c_1 & c_0
\end{bmatrix}.
$$
If the first column entries of a matrix shift up when moving right, as shown in the following
$$
\begin{bmatrix}
c_0 & c_1 & c_2 \\
c_1 & c_2 & c_0 \\
c_2 & c_0 & c_1
\end{bmatrix},
$$
then it is a special Hankel matrix. And we name it an
\emph{anti-circulant matrix}. Naturally, we generalize the
anti-circulant matrix to the tensor case. A square Hankel tensor
$\ten{C}$ of order $m$ and dimension  $n$ is called an
\emph{anti-circulant tensor}, if its generating vector ${\bf h}$
satisfies that
$$
h_{k} = h_{l}, \quad \text{if } k \equiv l \,({\rm mod}\, n).
$$
Thus the generating vector is periodic and displayed as
$$
{\bf h} = (\underbrace{h_0, h_1, \dots, h_{n-1}}_{{\bf c}^\top}, \underbrace{h_n, h_{n+1}, \dots, h_{2n-1}}_{{\bf c}^\top}, \dots, \underbrace{h_{(m-1)n}, \dots, h_{m(n-1)}}_{{\bf c}(0:n-m)^\top})^\top.
$$
Since the vector ${\bf c}$, which is exactly the ``first'' column $\ten{C}(:,0,\cdots,0)$, contains all the information about $\ten{C}$ and is more compact than the generating vector, we call it the \emph{compressed generating vector} of the anti-circulant tensor. For instance, a $3 \times 3 \times3$ anti-circulant tensor $\ten{C}$ is unfolded by mode-$1$ into
$$
{\rm Unfold}_1(\ten{C}) =
\left[\begin{array}{ccc}
c_0 & c_1 & c_2 \\
c_1 & c_2 & c_0 \\
c_2 & c_0 & c_1
\end{array}\right.
\left|\begin{array}{ccc}
c_1 & c_2 & c_0 \\
c_2 & c_0 & c_1 \\
c_0 & c_1 & c_2
\end{array}\right|
\left.\begin{array}{ccc}
c_2 & c_0 & c_1 \\
c_0 & c_1 & c_2 \\
c_1 & c_2 & c_0
\end{array}\right],
$$
and its compressed generating vector is ${\bf c} = [c_0,c_1,c_2]^\top$. Note that the degree of freedom of an anti-circulant tensor is always $n$ no matter how large its order $m$ will be.

\subsection{Diagonalization}

One of the essential properties of circulant matrices is that every circulant matrix can be diagonalized by the Fourier matrix, where the Fourier matrix of size $n$ is defined as $F_n = \big(\exp(-\frac{2\pi{\bf i}}{n}jk)\big)_{j,k=0,1,\dots,n-1}$. Actually, the Fourier matrix is exactly the Vandermonde matrix for the roots of unity, and it is also a unitary matrix up to the normalization factor
$$
F_n^{} F_n^\ast = F_n^\ast F_n^{} = n I_n^{},
$$
where $I_n$ is the identity matrix of $n \times n$ and $F_n^\ast$ is the conjugate transpose of $F_n$. We will show  that anti-circulant tensors also have a similar property, which brings much convenience for both analysis and computations.

In order to describe this property, we recall the definition of mode-$p$ tensor-matrix product first. It should be pointed out that the tensor-matrix products in this paper are slightly different with some standard notations (cf. \cite{de2000multilinear,kolda2009tensor}) just for easy use and simple descriptions. Let $\ten{A}$ be a tensor of $n_1 \times n_2 \times \dots \times n_m$ and $M$ be a matrix of $n_p \times n_p'$, then the mode-$p$ product $\ten{A} \times_p M$ is another tensor of $n_1 \times \dots \times n_{p-1} \times n_p' \times n_{p+1} \times \dots \times n_m$ defined as
$$
(\ten{A} \times_p M)_{i_1 \dots i_{p-1} j i_{p+1} \dots i_m} = \sum_{i_p=0}^{n_p-1} \ten{A}_{i_1 \dots i_{p-1} i_p i_{p+1} \dots i_m} M_{i_p j}.
$$
In the standard notation system, two indices in ``$M_{i_p j}$'' should be exchanged. There are some basic properties of the tensor-matrix products
\begin{enumerate}
  \item $\ten{A} \times_p M_p \times_q M_q = \ten{A} \times_q M_q \times_p M_p$, if $p \neq q$,
  \item $\ten{A} \times_p M_{p_1} \times_p M_{p_2} = \ten{A} \times_p (M_{p_1}M_{p_2})$,
  \item $\ten{A} \times_p M_{p_1} + \ten{A} \times_p M_{p_2} = \ten{A} \times_p (M_{p_1} + M_{p_2})$,
  \item $\ten{A}_1 \times_p M + \ten{A}_2 \times_p M = (\ten{A}_1+\ten{A}_2) \times_p M$.
\end{enumerate}
Particularly, when $\ten{A}$ is a matrix, the mode-$1$ and mode-$2$ products can be written as
$$
\ten{A} \times_1 M_1 \times_2 M_2 = M_1^\top \ten{A} M_2.
$$
Notice that this is totally different with $M_1^\ast \ten{A} M_2$! ($M_1^\ast$ is the conjugate transpose of $M_1$.) We will adopt  some notations from Qi's paper \cite{qi2013hankel}
\[
\begin{split}
\ten{A}{\bf x}^{m-1} &= \ten{A} \hspace{26pt} \times_2 {\bf x} \dots \times_m {\bf x}, \\
\ten{A}{\bf x}^{m} &= \ten{A} \times_1 {\bf x} \times_2 {\bf x} \dots \times_m {\bf x}.
\end{split}
\]

We are now ready to state our main result about anti-circulant tensors.
\begin{theorem}\label{circ_diag}
A square tensor of order $m$ and dimension  $n$ is an anti-circulant tensor if and only if it can be diagonalized by the Fourier matrix of size $n$, that is,
$$
\ten{C} = \ten{D}F_n^m := \ten{D} \times_1 F_n \times_2 F_n \dots \times_m F_n,
$$
where $\ten{D}$ is a diagonal tensor and ${\rm diag}(\ten{D}) = {\rm ifft}({\bf c})$. Here, ``${\rm ifft}$" is a Matlab-type symbol, an abbreviation of inverse fast Fourier transformation.
\end{theorem}
\begin{proof}
It is direct to verify that a tensor that can be expressed as $\ten{D}F_n^m$ is anti-circulant. Thus we only need to prove that every anti-circulant tensor can be written into this form. And this can be done constructively.

First, assume that an anti-circulant tensor $\ten{C}$ could be written into $\ten{D}F_n^m$. Then how do we obtain the diagonal entries of $\ten{D}$ from $\ten{C}$? Since
\[
\begin{split}
{\rm diag}(\ten{D})
&= \ten{D} {\bf 1}^{m-1} = \frac{1}{n^m} \big(\ten{C}(F_n^\ast)^m\big) {\bf 1}^{m-1} = \frac{1}{n^m} \bar{F}_n \big(\ten{C} (F_n^\ast {\bf 1})^{m-1}\big) \\
&= \frac{1}{n} \bar{F}_n (\ten{C} {\bf e}_0^{m-1}) = \frac{1}{n} \bar{F}_n {\bf c},
\end{split}
\]
where ${\bf 1} = [1,1,\dots,1]^\top$, ${\bf e}_0 = [1,0,\dots,0]^\top$, and ${\bf c}$ is the compressed generating vector of $\ten{C}$, then the diagonal entries of $\ten{D}$ can be computed by an inverse fast Fourier transformation (IFFT)
$$
{\rm diag}(\ten{D}) = {\rm ifft}({\bf c}).
$$
Finally, it is enough to check that $\ten{C} = \ten{D}F^m$ with ${\rm diag}(\ten{D}) = {\rm ifft}({\bf c})$ directly. Therefore, every anti-circulant tensor is diagonalized by the Fourier matrix of proper size.
\end{proof}

From the expression $\ten{C} = \ten{D}F_n^m$, we have a corollary about the spectra of anti-circulant tensors. The definitions of tensor Z-eigenvalues and H-eigenvalues follow the ones in \cite{qi2005eigenvalues,qi2007eigenvalues}.
\begin{corollary}
An anti-circulant tensor $\ten{C}$ of order $m$ and  dimension  $n$ with the compressed generating vector ${\bf c}$ has a Z-eigenvalue/H-eigenvalue $n^{m-2}{\bf 1}^\top {\bf c}$ with eigenvector ${\bf 1}$. When $n$ is even, it has another Z-eigenvalue $n^{m-2}\tilde{\bf 1}^\top {\bf c}$ with eigenvector $\tilde{\bf 1}$, where $\tilde{\bf 1} = [1,-1,\dots,1,-1]^\top$; Moreover, this is also an H-eigenvalue if $m$ is even.
\end{corollary}
\begin{proof}
Check it directly:
\[
\begin{split}
\ten{C}{\bf 1}^{m-1}
&= F_n^\top \big(\ten{D}(F_n{\bf 1})^{m-1}\big) = n^{m-1} F_n^\top (\ten{D}{\bf e}_0^{m-1}) = n^{m-1} \ten{D}_{1,1,\dots,1} \cdot F_n^\top {\bf e}_0 \\
&= n^{m-2} ({\bf e}_0^\top \bar{F}_n {\bf c}) {\bf 1} = (n^{m-2}{\bf 1}^\top {\bf c}) {\bf 1}.
\end{split}
\]
The proof of the rest part is similar, so we omit it.
\end{proof}

\subsection{Block Tensors}

Block structures arise in a variety of applications in scientific computing and engineering (cf. \cite{jin2002developments,Vadim2001}). And block matrices are successfully applied to many disciplines. Surely, we can also define block tensors. A \emph{block tensor} is understood as a tensor whose entries are also tensors. Let $\ten{A}$ be a block tensor. Then the $(j_1,j_2,\dots,j_m)$-th block is denoted as $\ten{A}^{(j_1 j_2 \dots j_m)}$, and the $(i_1,i_2,\dots,i_m)$-th entry of the $(j_1,j_2,\dots,j_m)$-th block is denoted as $\ten{A}_{i_1 i_2 \dots i_m}^{(j_1 j_2 \dots j_m)}$.

If a block tensor can be regarded as a Hankel tensor or an anti-circulant tensor with tensor entries, then we call it a \emph{block Hankel tensor} or a \emph{block anti-circulant tensor}, respectively. Moreover, its generating vector ${\bf h}^{(b)}$ or compressed generating vector ${\bf c}^{(b)}$ with tensor entries is called the \emph{block generating vector} or \emph{block compressed generating vector}, respectively. For instance, let $\ten{H}$ be a block Hankel tensor of size $(n_1 N_1) \times (n_2 N_2) \times \dots \times (n_m N_m)$ with blocks of size $n_1 \times n_2 \times \dots \times n_m$, then ${\bf h}^{(b)}$ is a block tensor of size $[(N_1+ N_2+ \cdots+N_m-m+1)n_1] \times n_2 \times \dots \times n_m$. Recall the definition of Kronecker product \cite{golub2012matrix}
$$
A \otimes B = \begin{bmatrix}
A_{11} B & A_{12} B & \cdots & A_{1 q} B \\
\vdots &\vdots & \ddots & \vdots \\
A_{p1} B & A_{p2} B & \cdots & A_{pq} B
\end{bmatrix},
$$
where $A$ and $B$ are two matrices of arbitrary sizes. Then it can be proved following Theorem \ref{circ_diag} that a block anti-circulant tensor $\ten{C}$ can be block-diagonalized by $F_N \otimes I$, that is,
$$
\ten{C} = \ten{D}^{(b)} (F_N \otimes I)^m,
$$
where $\ten{D}^{(b)}$ is a block diagonal tensor with diagonal blocks ${\bf c}^{(b)} \times_1 \big(\frac{1}{N}\bar{F}_N \otimes I\big)$ and
$\bar{F}_N$ is the conjugate of $F_N$.

Furthermore, when the blocks of a block Hankel tensor are also Hankel tensors, we call it a \emph{block Hankel tensor with Hankel blocks}, or BHHB tensor for short. Then its block generating vector can be reduced to a matrix, which is called the \emph{generating matrix} $H$ of a BHHB tensor
$$
H = [{\bf h}_0,{\bf h}_1,\dots,{\bf h}_{N_1+\dots+N_m-m}] \in \mathbb{C}^{(n_1+n_2+\dots+n_m-m+1)\times(N_1+N_2+\dots+N_m-m+1)},
$$
where ${\bf h}_k$ is the generating vector of the $k$-th Hankel block in ${\bf h}^{(b)}$. Similarly, when the blocks of a block anti-circulant tensor are also anti-circulant tensors, we call it a \emph{block anti-circulant tensor with anti-circulant blocks}, or BAAB tensor for short. Then its \emph{compressed generating matrix} $C$ is defined as
$$
C = [{\bf c}_0,{\bf c}_1,\dots,{\bf c}_{N-1}] \in \mathbb{C}^{n \times N},
$$
where ${\bf c}_k$ is the compressed generating vector of the $k$-th anti-circulant block in the block compressed generating vector ${\bf c}^{(b)}$. We can also verify that a BAAB tensor $\ten{C}$ can be diagonalized by $F_N \otimes F_n$, that is,
$$
\ten{C} = \ten{D} (F_N \otimes F_n)^m,
$$
where $\ten{D}$ is a diagonal tensor with diagonal ${\rm diag}(\ten{D}) = \frac{1}{nN}{\rm vec}(\bar{F}_n C\bar{F}_N)$, which can be computed by 2D inverse fast Fourier transformation (IFFT2). Here ${\rm vec}(\cdot)$ denotes the vectorization operator (cf. \cite{golub2012matrix}).

We can even define higher-level block Hankel tensors. For instance, a block Hankel tensor with BHHB blocks is called a \emph{level-$3$ block Hankel tensor}, and it is easily understood that a level-$3$ block Hankel tensor has the generating tensor of order $3$. Generally, a block Hankel or anti-circulant tensor with level-$(k$-$1)$ block Hankel or anti-circulant blocks is called a \emph{level-$k$ block Hankel or anti-circulant tensor}, respectively. Furthermore, a level-$k$ block anti-circulant tensor $\ten{C}$ can be diagonalized by $F_{n^{(k)}} \otimes F_{n^{(k-1)}} \otimes \dots \otimes F_{n^{(1)}}$, that is,
$$
\ten{C} = \ten{D} (F_{n^{(k)}} \otimes F_{n^{(k-1)}} \otimes \dots \otimes F_{n^{(1)}})^m,
$$
where $\ten{D}$ is a diagonal tensor with diagonal that can be computed by multi-dimensional inverse fast Fourier transformation.

\section{Fast Hankel Tensor-Vector Product}

General tensor-vector products without structures are very expensive
for the high order and the large size that a tensor could be of. For
a square tensor $\ten{A}$ of order $m$ and dimension  $n$, the
computational complexity of a tensor-vector product $\ten{A}{\bf
x}^{m-1}$ or $\ten{A}{\bf x}^m$ is $\ten{O}(n^m)$. However, since
Hankel tensors and anti-circulant tensors have very low degrees of
freedom, it can be expected that there is a much faster algorithm
for Hankel tensor-vector products. We focus on the following two
types of tensor-vector products
$$
{\bf y} = \ten{A} \times_2 {\bf x}_2 \cdots \times_m {\bf x}_m \text{ and } \alpha = \ten{A} \times_1 {\bf x}_1 \times_2 {\bf x}_2 \cdots \times_m {\bf x}_m,
$$
which will be extremely useful to applications.

The fast algorithm for anti-circulant tensor-vector products is easy to derive from Theorem \ref{circ_diag}. Let $\ten{C} = \ten{D}F_n^m$ be an anti-circulant tensor of order $m$ and dimension  $n$ with the compressed generating vector ${\bf c}$. Then for vectors ${\bf x}_2, {\bf x}_3, \dots, {\bf x}_m \in \mathbb{C}^n$, we have
$$
{\bf y} = \ten{C} \times_2 {\bf x}_2 \cdots \times_m {\bf x}_m = F_n (\ten{D} \times_2 F_n{\bf x}_2 \cdots \times_m F_n{\bf x}_m).
$$
Recall that ${\rm diag}(\ten{D}) = {\rm ifft}({\bf c})$ and $F_n {\bf v} = {\rm fft}({\bf v})$, where ``${\rm fft}$" is a Matlab-type symbol, an abbreviation of fast Fourier transformation. So the fast procedure for computing the vector ${\bf y}$ is
$$
{\bf y} = {\rm fft}\big( {\rm ifft}({\bf c}) .\ast {\rm fft}({\bf x}_2) .\ast \cdots .\ast {\rm fft}({\bf x}_m)\big),
$$
where ${\bf u}.\ast{\bf v}$ multiplies  two vectors element-by-element. Similarly, for vectors ${\bf x}_1,{\bf x}_2,\dots,{\bf x}_m \in \mathbb{C}^n$, we have
$$
\alpha = \ten{C} \times_1 {\bf x}_1 \times_2 {\bf x}_2 \cdots \times_m {\bf x}_m = \ten{D} \times_1 F_n{\bf x}_1 \times_2 F_n{\bf x}_2 \cdots \times_m F_n{\bf x}_m,
$$
and the fast procedure for computing the scalar $\alpha$ is
$$
\alpha = {\rm ifft}({\bf c})^\top \big( {\rm fft}({\bf x}_1) .\ast {\rm fft}({\bf x}_2) .\ast \cdots .\ast {\rm fft}({\bf x}_m)\big).
$$
Since the computational complexity of either FFT or IFFT is $\ten{O}(n \log n)$, both two types of anti-circulant tensor-vector products can be obtained with complexity $\ten{O}\big((m+1)n \log n\big)$, which is much faster than the product of a general $n$-by-$n$ matrix with a vector.

For deriving the fast algorithm for Hankel tensor-vector products, we embed a Hankel tensor into a larger anti-circulant tensor. Let $\ten{H} \in \mathbb{C}^{n_1 \times n_2 \times \dots \times n_m}$ be a Hankel tensor with the generating vector ${\bf h}$. Denote $\ten{C}_\ten{H}$ as the anti-circulant tensor of order $m$ and dimension  $d_\ten{H}=n_1+n_2+\dots+n_m-m+1$ with the compressed generating vector ${\bf h}$. Then we will find out that $\ten{H}$ is in the ``upper left frontal'' corner of $\ten{C}_\ten{H}$ as shown in Figure \ref{embed}.
\begin{figure}[t]
  \centering
  \includegraphics[width=250pt]{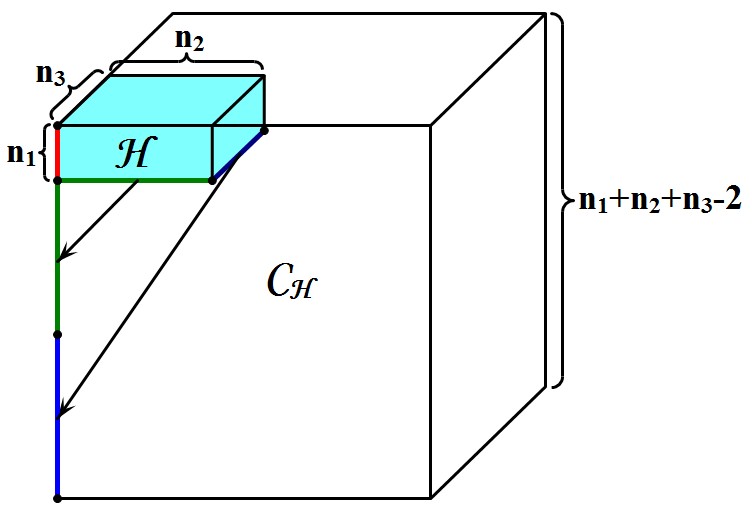}
  \caption{Embed a Hankel tensor into an anti-circulant tensor.}\label{embed}
\end{figure}
Hence, we have
\[
\begin{split}
&\ten{C}_\ten{H} \times_2 \begin{bmatrix} {\bf x}_2 \\ {\bf 0} \end{bmatrix} \dots \times_m \begin{bmatrix} {\bf x}_m \\ {\bf 0} \end{bmatrix} = \begin{bmatrix} \ten{H} \times_2 {\bf x}_2 \dots \times_m {\bf x}_m \\ {\bf \dag} \end{bmatrix}, \\
&\ten{C}_\ten{H} \times_1 \begin{bmatrix} {\bf x}_1 \\ {\bf 0} \end{bmatrix} \dots \times_m \begin{bmatrix} {\bf x}_m \\ {\bf 0} \end{bmatrix} = \ten{H} \times_1 {\bf x}_1 \dots \times_m {\bf x}_m,
\end{split}
\]
so that the Hankel tensor-vector products can be realized by multiplying a larger anti-circulant tensor by some augmented vectors. Therefore, the fast procedure for computing ${\bf y} = \ten{H} \times_2 {\bf x}_2 \dots \times_m {\bf x}_m$ is
$$
\left\{
\begin{array}{l}
\tilde{\bf x}_p = [{\bf x}_p^\top,\underbrace{0,0,\dots,0}_{d_\ten{H}-n_p}]^\top,\ p = 2,3,\dots,m, \\
\tilde{\bf y} = {\rm fft}\big( {\rm ifft}({\bf h}) .\ast {\rm fft}(\tilde{\bf x}_2) .\ast \cdots .\ast {\rm fft}(\tilde{\bf x}_m)\big), \\
{\bf y} = \tilde{\bf y}(0:n_1-1),
\end{array}
\right.
$$
and the fast procedure for computing $\alpha = \ten{H} \times_1 {\bf x}_1 \times_2 {\bf x}_2 \dots \times_m {\bf x}_m$ is
$$
\left\{
\begin{array}{l}
\tilde{\bf x}_p = [{\bf x}_p^\top,\underbrace{0,0,\dots,0}_{d_\ten{H}-n_p}]^\top,\ p = 1,2,\dots,m, \\
\alpha = {\rm ifft}({\bf h})^\top \big({\rm fft}(\tilde{\bf x}_1) .\ast {\rm fft}(\tilde{\bf x}_2) .\ast \cdots .\ast {\rm fft}(\tilde{\bf x}_m)\big).
\end{array}
\right.
$$
Moreover, the computational complexity is
$\ten{O}\big((m+1)d_\ten{H} \log d_\ten{H}\big)$. When the Hankel
tensor is a square tensor, the complexity is at the level
$\ten{O}(m^2n \log mn)$, which is much smaller than the complexity
$\ten{O}(n^m)$ of non-structured products.

Apart from the low computational complexity, our algorithm for Hankel tensor-vector products has two advantages. One is that this scheme is compact, that is, there is no redundant element in the procedure. It is not required to form the Hankel tensor explicitly. Just the generating vector is needed. Another advantage is that our algorithm treats the tensor as an ensemble instead of multiplying the tensor by vectors mode by mode.

\begin{example}
We choose $3^{\rm th}$-order square Hankel tensors $\ten{H}$ of
different sizes. We compute the tensor-vector products
$\ten{H}{\bf x}^{m-1}$ by using our fast algorithm and the
non-structured algorithm directly based on the definition. The
comparison of the running times is shown in Figure \ref{product}.
From the results, we can see that the running time of our algorithm
increases far slowly than that of the non-structured algorithm just
as the theoretical analysis. Moreover, the difference in
running times is not only  the low computational complexity,
but also  the absence of forming the Hankel tensor explicitly
in our algorithm.

\begin{figure}[t]
  \centering
  \includegraphics[width=390pt]{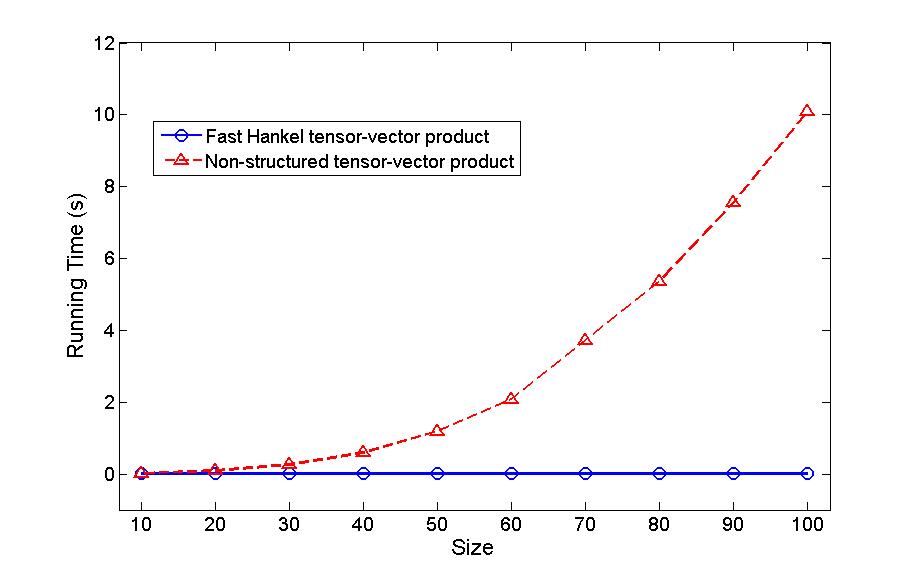}
  \caption{The running times of 100 products for each size and algorithm.}\label{product}
\end{figure}
\end{example}

For BAAB and BHHB cases, we also have fast algorithms for the tensor-vector products. Let $\ten{C}$ be a BAAB tensor of order $m$ with the compressed generating matrix $C \in \mathbb{C}^{n \times N}$. Since $\ten{C}$ can be diagonalized by $F_N \otimes F_n$, i.e.,
$$
\ten{C} = \ten{D} (F_N \otimes F_n)^m,
$$
we have for vectors ${\bf x}_2,{\bf x}_3,\dots,{\bf x}_m \in \mathbb{C}^{nN}$
$$
{\bf y} = \ten{C} \times_2 {\bf x}_2 \cdots \times_m {\bf x}_m = (F_N \otimes F_n) \big(\ten{D} \times_2 (F_N \otimes F_n){\bf x}_2 \cdots \times_m (F_N \otimes F_n){\bf x}_m\big).
$$
Recall the vectorization operator and its inverse operator
\[
\begin{split}
{\rm vec}(A) &= [A_{:,0}^\top,A_{:,1}^\top,\dots,A_{:,N-1}^\top]^\top \in \mathbb{C}^{nN}, \\
{\rm vec}_{n,N}^{-1}({\bf v}) &= [{\bf v}_{0:n-1},{\bf v}_{n:2n-1},\dots,{\bf v}_{(N-1)n:Nn-1}] \in \mathbb{C}^{n \times N},
\end{split}
\]
for matrix $A \in \mathbb{C}^{n \times N}$ and vector ${\bf v} \in \mathbb{C}^{nN}$, and the relation holds
$$
(B \otimes A){\bf v} = {\rm vec}\big(A \cdot {\rm vec}_{n,N}^{-1}({\bf v}) \cdot B^\top\big).
$$
So $(F_N \otimes F_n){\bf x}_p = {\rm vec}\big(F_n \cdot {\rm vec}_{n,N}^{-1}({\bf x}_p) \cdot F_N\big)$ can be computed by 2D fast Fourier transformation (FFT2). Then the fast procedure for computing ${\bf y} = \ten{C} \times_2 {\bf x}_2 \dots \times_m {\bf x}_m$ is
$$
\left\{
\begin{array}{l}
X_p = {\rm vec}_{n,N}^{-1}({\bf x}_p),\ p = 2,3,\dots,m, \\
Y = {\rm fft2}\big( {\rm ifft2}(C) .\ast {\rm fft2}(X_2) .\ast \cdots .\ast {\rm fft2}(X_m)\big), \\
{\bf y} = {\rm vec}(Y),
\end{array}
\right.
$$
and the fast procedure for computing $\alpha = \ten{C} \times_1 {\bf x}_1 \times_2 {\bf x}_2 \dots \times_m {\bf x}_m$ is
$$
\left\{
\begin{array}{l}
X_p = {\rm vec}_{n,N}^{-1}({\bf x}_p),\ p = 1,2,\dots,m, \\
\alpha = \big\langle{\rm ifft2}(C),{\rm fft2}(X_1) .\ast {\rm fft2}(X_2) .\ast \cdots .\ast {\rm fft2}(X_m)\big\rangle,
\end{array}
\right.
$$
where $\langle A,B \rangle$ denotes
$$
\langle A,B \rangle = \sum_{j,k} A_{jk} B_{jk}.
$$

For a BHHB tensor $\ten{H}$ with the generating matrix $H$, we do the embedding twice. First we embed each Hankel block into a larger anti-circulant block, and then we embed the block Hankel tensor with anti-circulant blocks into a BAAB tensor $\ten{C}_\ten{H}$. Notice that the compressed generating matrix of $\ten{C}_\ten{H}$ is exactly the generating matrix of $\ten{H}$. Hence we have the fast procedure for computing ${\bf y} = \ten{H} \times_2 {\bf x}_2 \dots \times_m {\bf x}_m$
$$
\left\{
\begin{array}{l}
\tilde{X}_p =\underbrace{\begin{bmatrix} {\rm vec}_{n_p,N_p}^{-1}({\bf x}_p) & {\rm O} \\ {\rm O} & {\rm O} \end{bmatrix}}_{N_1+N_2+\cdots+N_m-m+1} \hspace{-8pt}\left.\begin{matrix} \ \\ \ \end{matrix}\right\}{}^{n_1+n_2+ \cdots+n_m-m+1},\ p = 2,3,\dots,m, \\
\tilde{Y} = {\rm fft2}\big( {\rm ifft2}(H) .\ast {\rm fft2}(\tilde{X}_2) .\ast \cdots .\ast {\rm fft2}(\tilde{X}_m)\big), \\
{\bf y} = {\rm vec}\big(\tilde{Y}(0:n_1-1,0:N_1-1)\big).
\end{array}
\right.
$$
Sometimes in applications there is no need to do the vectorization in the last line, and we just keep it as a matrix for later use. We also have the fast procedure for computing $\alpha = \ten{H} \times_1 {\bf x}_1 \times_2 {\bf x}_2 \dots \times_m {\bf x}_m$
$$
\left\{
\begin{array}{l}
\tilde{X}_p =\underbrace{\begin{bmatrix} {\rm vec}_{n_p,N_p}^{-1}({\bf x}_p) & {\rm O} \\ {\rm O} & {\rm O} \end{bmatrix}}_{N_1+N_2+\cdots+N_m-m+1} \hspace{-8pt}\left.\begin{matrix} \ \\ \ \end{matrix}\right\}{}^{n_1+n_2+\cdots+n_m-m+1},\ p = 2,3,\dots,m, \\
\alpha = \big\langle{\rm ifft2}(H),{\rm fft2}(\tilde{X}_1) .\ast {\rm fft2}(\tilde{X}_2) .\ast \cdots .\ast {\rm fft2}(\tilde{X}_m)\big\rangle.
\end{array}
\right.
$$
Similarly, we can also derive the fast algorithms for higher-level block Hankel tensor-vector products using the multi-dimensional FFT.

\section{Exponential Data Fitting}

Exponential data fitting is very important in many applications in scientific computing and engineering, which represents the signals as a sum of exponentially damped sinusoids. The computations and applications of exponential data fitting are generally studied, and the reader  is interested in these topics can refer to \cite{pereyra2010exponential}. Furthermore, Papy et al. introduced a higher-order tensor approach into exponential data fitting in \cite{papy2005exponential,papy2009exponential} by connecting it with the Vandermonde decomposition of a Hankel tensor. And they showed that the tensor approach can do better than the classical matrix approach.

\subsection{Hankel Tensor Approach}

Assume that we get a one-dimensional noiseless signal with $N$ complex samples $\{x_n\}_{n=0}^{N-1}$, and this signal is modelled as a sum of $K$ exponentially damped complex sinusoids, i.e.,
$$
x_n = \sum_{k=1}^K a_k \exp({\bf i}\varphi_k) \exp\big((-\alpha_k+{\bf i}\omega_k) n \Delta t\big),
$$
where ${\bf i} = \sqrt{-1}$, $\Delta t$ is the sampling interval, and the amplitudes $a_k$, the phases $\varphi_k$, the damping factors $\alpha_k$, and the pulsations $\omega_k$ are the parameters of the model. The signal can also be expressed as
$$
x_n = \sum_{k=1}^K c_k z_k^n,
$$
where $c_k = a_k \exp({\bf i}\varphi_k)$ and $z_k = \exp\big((-\alpha_k+{\bf i}\omega_k) \Delta t\big)$. Here $c_k$ is called the $k$-th complex amplitude including the phase, and $z_k$ is called the $k$-th pole of the signal. A part of the aim of exponential data fitting is to estimate the poles $\{z_k\}_{k=1}^K$ from the data $\{x_n\}_{n=0}^{N-1}$.

Denote vector ${\bf x} = (x_n)_{n=0}^{N-1}$. Then we construct a
Hankel tensor $\ten{H}$ of a fixed order $m$ and size $I_1 \times
I_2 \times \dots \times I_m$ with the generating vector ${\bf x}$.
The order $m$ can be chosen arbitrarily and the size $I_p$ of each
dimension should be no less than $K$. And we always deal with
$3$-order tensors in this paper. Moreover, it is advised to
construct a Hankel tensor as close to a square tensor as well (cf.
\cite{papy2005exponential,papy2009exponential,pereyra2010exponential}).
Papy et al. verified that the Vandermonde decomposition of
$\ten{H}$ is
$$
\ten{H} = \ten{C} \times_1 Z_1^\top \times_2 Z_2^\top \dots \times_m Z_m^\top,
$$
where $\ten{C}$ is a diagonal tensor with diagonal entries $\{c_k\}_{k=1}^K$, and each $Z_p$ is a Vandermonde matrix
$$
Z_p^\top = \begin{bmatrix}
1 & z_1 & z_1^2 & \cdots & z_1^{I_p-1} \\
1 & z_2 & z_2^2 & \cdots & z_2^{I_p-1} \\
\vdots & \vdots & \vdots & \vdots & \vdots \\
1 & z_K & z_K^2 & \cdots & z_K^{I_p-1} \\
\end{bmatrix}.
$$
So the target will be attained, if we obtain the Vandermonde decomposition of this Hankel tensor $\ten{H}$. And the Vandermonde matrix can be estimated by applying the total least square (TLS, cf. \cite{golub2012matrix}) to the factor matrix in the higher-order singular value decomposition (HOSVD, cf. \cite{de2000multilinear,kolda2009tensor}) of the best rank-$(K,K,\dots,K)$ approximation (cf. \cite{de2000best,kolda2009tensor}) of $\ten{H}$. Therefore, computing the HOSVD of the best low rank approximation of a Hankel tensor is a main part in exponential data fitting.

\subsection{Fast HOOI for Hankel Tensors}

Now we have a look at the best rank-$(R_1,R_2,\dots,R_m)$
approximation of a tensor. This concept was first introduced by De
Lathauwer et al. in \cite{de2000best}, and they also proposed an
effective algorithm called higher-order orthogonal iterations
(HOOI). There are  other algorithms with faster convergence such
as \cite{derayleigh} and \cite{elden2009newton} proposed, and one
can refer to \cite{kolda2009tensor} for more details. However, HOOI
is still very popular, because it is so simple and still effective in
applications. Papy et al. chose HOOI in \cite{papy2005exponential},
so we aim at modifying the general HOOI into a specific algorithm
for Hankel tensors in order to get higher speed for exponential data
fitting. And we will focus on the square Hankel tensor, since it is
preferred in exponential data fitting.

The original HOOI algorithm for general tensors is as the following.
\begin{algorithm}
HOOI for computing the best rank-$(R_1,R_2,\dots,R_m)$ approximation of tensor $\ten{A} \in \mathbb{C}^{I_1 \times I_2 \times \dots \times I_m}$.

\parbox[t]{360pt}{\rm
Initialize $U_p \in \mathbb{C}^{I_p \times R_p}$ $(p = 1:m)$ by HOSVD of $\ten{A}$ \\
Repeat \\
\hspace*{15pt} for $p = 1:m$ \\
\hspace*{35pt} $U_p \leftarrow R_p$ leading left singular vectors of \\
\hspace*{68pt} ${\rm Unfold}_p(\ten{A} \times_1 \bar{U}_1 \cdots \times_{p-1} \bar{U}_{p-1} \times_{p+1} \bar{U}_{p+1} \cdots \times_m \bar{U}_m)$ \\
\hspace*{15pt} end \\
Until convergence \\
$\ten{S} = \ten{A} \times_1 \bar{U}_1 \times_2 \bar{U}_2 \cdots \times_m \bar{U}_m$.
}
\end{algorithm}

As employing this algorithm to Hankel tensors, we have to do some modifications to involve the structures. First, the computation of HOSVD of a Hankel tensor  can be more efficient. The HOSVD is obtained by computing the SVD of every unfolding of $\ten{A}$,  Badeau and Boyer in \cite{badeau2008fast} pointed out that many columns in the unfolding of a structured tensor are redundant, so that we can remove them for a smaller scale. Particularly, the mode-$p$ unfolding of a Hankel tensor $\ten{H}$ after removing the redundant columns is a rectangular Hankel matrix of size $I_p \times (d_\ten{H}-I_p+1)$ with the same generating vector with $\ten{H}$
$$
{\rm Unfold}_p(\ten{H}) \rightarrow
\begin{bmatrix}
h_0 & h_1 & \cdots & h_{d_\ten{H}-I_p-1} & h_{d_\ten{H}-I_p} \\
h_1 & \iddots & \iddots & \iddots & \vdots \\
\vdots & \iddots & \iddots & \iddots & h_{d_\ten{H}-2}  \\
h_{I_p-1} & h_{I_p} & \cdots & h_{d_\ten{H}-2} & h_{d_\ten{H}-1}
\end{bmatrix}.
$$
The SVD for Hankel matrices (called Symmetric SVD or Takagi Factorization) can be computed more efficiently than for general matrices by using the algorithm in \cite{browne2009lanczos,xu2008fast}. This is the first point that can be modified.

Next, there are plenty of Hankel tensor-matrix products in the above algorithm, and they can be complemented by using our fast Hankel tensor-vector products. For instance, the tensor-matrix product
$$
(\ten{H} \times_2 \bar{U}_2 \dots \times_m \bar{U}_m)_{:,i_2,\dots,i_m} = \ten{H} \times_2 (\bar{U}_2)_{:,i_2} \dots \times_m (\bar{U}_m)_{:,i_m},
$$
and others are the same. So all the Hankel tensor-matrix products  can be computed by our fast algorithm.

Finally, if the Hankel tensor that we construct is further a square
tensor, then we can expect to obtain a \emph{higher-order symmetric
singular value decomposition} or called a \emph{higher-order Takagi
factorization} (HOTF), that is, $U_1=U_2=\dots=U_m$. Hence we do not
need to update them separately in each step. Therefore, the HOOI
algorithm modified for square Hankel tensors is as follows.
\begin{algorithm}
HOOI for computing the best rank-$(R,R,\dots,R)$ approximation of
square Hankel tensor $\ten{H} \in \mathbb{C}^{I \times I \times
\dots \times I}$.

\parbox[t]{360pt}{\rm
$U \leftarrow R$ leading left singular vectors of ${\rm hankel}({\bf h}_{0:I-1},{\bf h}_{I-1:nI-m})$ \\
Repeat \\
\hspace*{15pt} $U \leftarrow R$ leading left singular vectors of ${\rm Unfold}_1(\ten{H} \bar{U}^{m-1})$ \\
Until convergence \\
$\ten{S} = \ten{H} \bar{U}^m$.
}
\end{algorithm}

\subsection{2-Dimensional Exponential Data Fitting}

Many real-world signals are multi-dimensional, and the exponential model is also useful for these problems, such that 2D curve fitting, 2D signal recovery, etc. The multi-dimensional signals are certainly more complicated than the one-dimensional signals. Thus one Hankel tensor is never enough for representing the signals. So we need multi-level block Hankel tensors for multi-dimensional signals. We take the 2D exponential data fitting for an example to illustrate our block approach.

Assume that there is a 2D noiseless signal with $N_1 \times N_2$ complex samples $\{x_{n_1n_2}\}_{\begin{subarray}{l} n_1=0,1,\dots,N_1-1 \\ n_2 = 0,1,\dots,N_2-1 \end{subarray}}$ which is modelled as a sum of $K$ exponential items
$$
x_{n_1n_2} = \sum_{k=1}^K a_k \exp({\bf i}\varphi_k) \exp\big((-\alpha_k+{\bf i}\omega_k)n_1 \Delta t_1 + (-\beta_k+{\bf i}\nu_k)n_2 \Delta t_2\big),
$$
where the meanings of  parameters are the same as those of 1D signals. Also, the 2D signal can be rewritten into a compact form
$$
x_{n_1n_2} = \sum_{k=1}^K c_k z_{1,k}^{n_1} z_{2,k}^{n_2}.
$$
And our aim is still to estimate the poles $\{z_{1,k}\}$ and $\{z_{2,k}\}$ of the signal from the samples $\{x_{n_1n_2}\}_{\begin{subarray}{l} n_1=0,1,\dots,N_1-1 \\ n_2 = 0,1,\dots,N_2-1 \end{subarray}}$ that we receive.

Denote matrix $X = (x_{n_1 n_2})_{N_1 \times N_2}$. Then we construct a BHHB tensor $\ten{H}$ with the generating matrix $X$ of a fixed order $m$, size $(I_1 J_1) \times (I_2 J_2) \times \dots \times (I_m J_m)$, and block size $I_1 \times I_2 \times \dots \times I_m$. The order $m$ can be selected arbitrarily and the size $I_p$ and $J_p$ of each dimension should be no less than $K$. Then the BHHB tensor $\ten{H}$ has the level-2 Vandermonde decomposition
$$
\ten{H} = \ten{C} \times_1 (Z_{2,1} \oslash Z_{1,1})^\top \times_2 (Z_{2,2} \oslash Z_{1,2})^\top \dots \times_m (Z_{2,m} \oslash Z_{1,m})^\top,
$$
where $\ten{C}$ is a diagonal tensor with diagonal entries $\{c_k\}_{k=1}^K$, each $Z_{1,p}$ or $Z_{2,p}$ is a Vandermonde matrix
$$
Z_{1,p}^\top = \begin{bmatrix}
1 & z_{1,1} & z_{1,1}^2 & \cdots & z_{1,1}^{I_p-1} \\
1 & z_{1,2} & z_{1,2}^2 & \cdots & z_{1,2}^{I_p-1} \\
\vdots & \vdots & \vdots & \vdots & \vdots \\
1 & z_{1,K} & z_{1,K}^2 & \cdots & z_{1,K}^{I_p-1} \\
\end{bmatrix},\
Z_{2,p}^\top = \begin{bmatrix}
1 & z_{2,1} & z_{2,1}^2 & \cdots & z_{2,1}^{J_p-1} \\
1 & z_{2,2} & z_{2,2}^2 & \cdots & z_{2,2}^{J_p-1} \\
\vdots & \vdots & \vdots & \vdots & \vdots \\
1 & z_{2,K} & z_{2,K}^2 & \cdots & z_{2,K}^{J_p-1} \\
\end{bmatrix},
$$
and the notation ``$\oslash$" denotes the column-wise Kronecker product of two matrices with the same column sizes, i.e.,
$$
(A \oslash B)(:,j) = A(:,j) \otimes B(:,j).
$$
So the target will be attained, if we obtain the level-2 Vandermonde decomposition of this BHHB tensor $\ten{H}$.

We can use HOOI as well to compute the best rank-$(K,K,\dots,K)$ approximation of the BHHB tensor $\ten{H}$
$$
\ten{H} = \ten{S} \times_1 U_1^\top \times_2 U_2^\top \dots \times_m U_m^\top,
$$
where $\ten{S} \in \mathbb{C}^{K \times K \times \dots \times K}$ is the core tensor and $U_p \in \mathbb{C}^{(I_p J_p) \times K}$ has orthogonal columns. Note that the tensor-vector products in HOOI should be computed by our fast algorithm for BHHB tensor-vector products in Section 3. Then $U_p$ and $Z_{2,p} \oslash Z_{1,p}$ have the common column space, which is equivalent to that there is a nonsingular matrix $T$ such that
$$
Z_{2,p} \oslash Z_{1,p} = U_p T.
$$
Denote
\[
\begin{array}{l}
A^{1\uparrow} = \big[A_{0:I-2,:}^\top,A_{I,2I-2,:}^\top,\dots,A_{(J-1)I:JI-2,:}^\top\big]^\top, \\
A^{1\downarrow} = \big[A_{1:I-1,:}^\top,A_{I+1,2I-1,:}^\top,\dots,A_{(J-1)I+1:JI-1,:}^\top\big]^\top, \\
A^{2\uparrow} = A_{0:(J-1)I-1,:}, \\
A^{2\downarrow} = A_{I:JI-1,:},
\end{array}
\]
for matrix $A \in \mathbb{C}^{(IJ) \times K}$. Then it is easy to verify that
\[
\begin{split}
(Z_{2,p} \oslash Z_{1,p})^{1\uparrow} D_1  &= (Z_{2,p} \oslash Z_{1,p})^{1\downarrow}, \\
(Z_{2,p} \oslash Z_{1,p})^{2\uparrow} D_2  &= (Z_{2,p} \oslash Z_{1,p})^{2\downarrow},
\end{split}
\]
where $D_1$ is a diagonal matrix with diagonal entries $\{z_{1,k}\}_{k=1}^K$ and $D_2$ is a diagonal matrix with diagonal entries $\{z_{2,k}\}_{k=1}^K$. Then we have
\[
\begin{split}
U_p^{1\uparrow} (T D_1 T^{-1}) &= U_p^{1\downarrow}, \\
U_p^{2\uparrow} (T D_2 T^{-1}) &= U_p^{2\downarrow}.
\end{split}
\]
Therefore, if we obtain two matrices $W_1$ and $W_2$ satisfying that
\[
\begin{split}
U_p^{1\uparrow} W_1 &= U_p^{1\downarrow}, \\
U_p^{2\uparrow} W_2 &= U_p^{2\downarrow},
\end{split}
\]
then $W_1$ and $W_2$ share the same eigenvalues with $D_1$ and $D_2$, respectively. Equivalently, the eigenvalues of $W_1$ and $W_2$ are exactly the poles of the first and second dimension, respectively. Furthermore, we also choose the total least square as in \cite{papy2005exponential} for solving the above two equations, since the noise on both sides should be taken into consideration.

\begin{example}
We will show the effectiveness by a $2$-dimensional $2$-peak example partly from \cite{papy2005exponential},
\[
\begin{split}
y_{n_1 n_2} &= x_{n_1 n_2} + e_{n_1 n_2} \\
&= \exp\big((-0.01+2\pi{\bf i}0.20)n_1\big) \cdot \exp\big((-0.02+2\pi{\bf i}0.18)n_2\big) \\
&\hspace{1pt}+ \exp\big((-0.02+2\pi{\bf i}0.22)n_1\big) \cdot \exp\big((-0.01-2\pi{\bf i}0.20)n_2\big) + e_{n_1 n_2},
\end{split}
\]
where the first dimension of this signal is the same as the example in \cite{papy2005exponential} and $e_{n_1 n_2}$ is the noise. Since this signal has two peaks, the BHHB tensor constructed following $\{x_{n_1 n_2}\}$ is supposed to be rank-2. Figure \ref{2DSV} illustrates the variation of the mode-1 singular values of a BHHB tensor of size $6 \times 6 \times 6$ and block size $3 \times 3 \times 3$ with different levels of noise. When the signal is noiseless, there are only two nonzero mode-1 singular values, and others are about the machine epsilon, thus is regarded as zero. When the noise is added on, the ``zero singular values'' are at the same level with the noise, and those two largest singular values can still be separated easily from the others.
\begin{figure}[t]
  \centering
  \includegraphics[width=320pt]{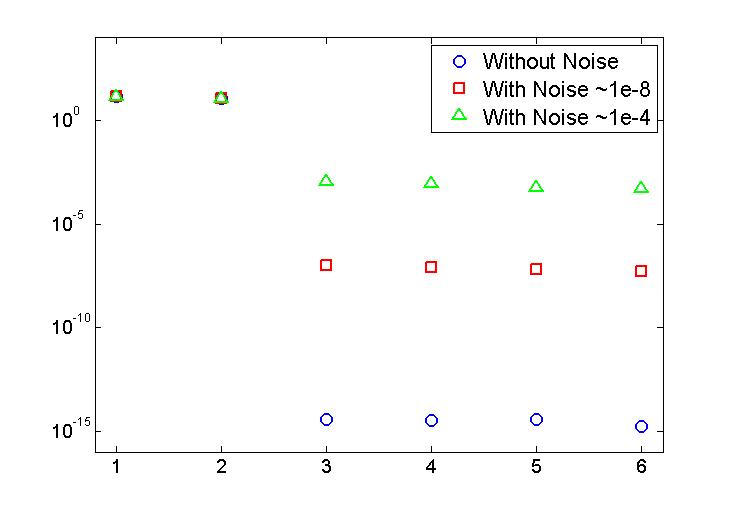}
  \caption{The mode-1 singular values of the constructed BHHB tensor.}\label{2DSV}
\end{figure}

Then we compute the HOTF of the best rank-(2,2,2) approximation of a BHHB tensor of size $30 \times 30 \times 30$ and block size $6 \times 6 \times 6$, and apply the total least square twice to estimate the poles, as described previously. Since we know the exact solution
$$
\begin{array}{ll}
z_{11} = \exp(-0.01+2\pi{\bf i}0.20), & z_{12} = \exp(-0.02+2\pi{\bf i}0.22), \\
z_{21} = \exp(-0.02+2\pi{\bf i}0.18), & z_{22} = \exp(-0.01-2\pi{\bf i}0.20),
\end{array}
$$
the relative errors are shown in Figure \ref{2DError}. We conclude that the error can attain the same level with the noise by employing our algorithm for 2D exponential data fitting.
\begin{figure}[t]
  \centering
  \includegraphics[width=300pt]{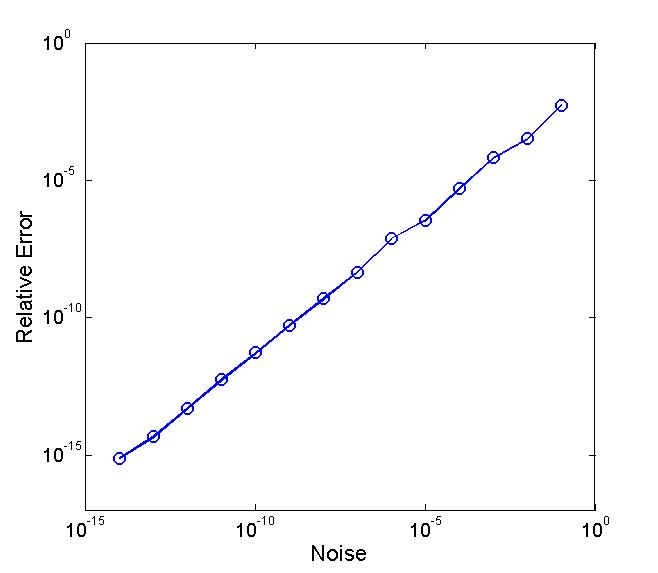}
  \caption{The relative errors with different levels of noise.}\label{2DError}
\end{figure}
\end{example}

\section{Conclusions}

We propose a fast algorithm for Hankel tensor-vector products with computational complexity $\ten{O}\big((m+1) d_\ten{H} \log d_\ten{H}\big)$, where $d_\ten{H} = n_1+n_2+\dots+n_m-m+1$. This fast algorithm is derived by embedding the Hankel tensor into a larger anti-circulant tensor, which can be diagonalized by the Fourier matrix. The fast algorithm for block Hankel tensors with Hankel blocks is also described. Furthermore, the fast Hankel and BHHB tensor-vector products are applied to HOOI in 1D and 2D exponential data fitting, respectively. The numerical experiments show the efficiency and effectiveness of our algorithms. Finally, this fast scheme should be introduced into every algorithm which involves Hankel tensor-vector products to improve its performance.

\section*{Acknowledgements}
Weiyang Ding would like to thank Prof. Sanzheng Qiao for the useful discussions on fast algorithms for Hankel matrices. We also thank Professors Lars Eld\'{e}n and Michael K. Ng for their detailed comments.

{\small
\bibliographystyle{siam}
\nocite{*}
\bibliography{hankel}
}

\end{document}